\documentclass{amsart}
\usepackage{amsmath}
\usepackage{amsfonts}

\newtheorem{theorem}{Theorem}
\theoremstyle{plain}

\newtheorem{corollary}{Corollary}

\newtheorem{definition}{Definition}
\newtheorem{example}{Example}

\newtheorem{proposition}{Proposition}
\newtheorem{remark}{Remark}

\numberwithin{equation}{section}
\input{tcilatex}

\begin{document}
\title[Spectral theory for $RO^{\ast }$-algebras]{On the Spectral theory for
Rickart Ordered $^{\ast }$-algebras }
\author{Dmitry Sh. Goldstein}
\address{Dmitry Sh. Goldstein, Holon Academic Institute of Technology, 52
Golomb St., P.O. Box 305, Holon 58102, Israel}
\email{dmitryg@hit.ac.il}
\author{Alexander A. Katz}
\address{Alexander A. Katz, Department of Mathematics and Computer Science,
300 Howard Ave., DaSilva Hall 314, Staten Island, NY 10301, USA}
\email{katza@stjohns.edu}
\author{Roman B. Shklyar}
\address{Roman B. Shklyar, Holon Academic Institute of Technology, 52 Golomb
St., P.O. Box 305, Holon 58102, Israel}
\email{rshklyar@macs.biu.ac.il}
\date{July 12, 2006}
\subjclass[2000]{Primary 46K05, 46L05; Secondary 06F25}
\keywords{Rickart $C^{\ast }$-algebras, Rickart Ordered $^{\ast }$-algebras, 
$RO^{\ast }$-algebras, $C^{\ast }$-algebras}
\thanks{This paper is in final form and no version of it will be submitted
for publication elsewhere.}

\begin{abstract}
$RO^{\ast }$-algebras are defined and studied. For $RO^{\ast }$-algebra $T$
, using properties of partial order, it is established that the set of
bounded elements can be endowed with $C^{\ast }$-norm. The structure of
commutative subalgebras of $T$ is considered and the Spectral Theorem for
any self-adjoint element of $T$ is proven.
\end{abstract}

\maketitle

\section{Introduction}

The theory of $AW^{\ast }$-algebras and Baer $^{\ast }$-ring was introduced
and intensively studied by Kaplansky (see \cite{Kaplansky51} and \cite%
{Kaplansky68}). Those rings are defined axiomatically by an annihilator
condition and have a rich algebraic structure.

Let us recall some definitions:

\begin{definition}
A Baer $^{\ast }$-ring is a $^{\ast }$-ring\ $A$ such that for any subset $S$
of $A$, the right annihilator of $S$ is the principle right ideal generated
by a projection (a self-adjoint idempotent).
\end{definition}

The projections in Baer $^{\ast }$-rings form a lattice and admit a
classical equivalence relation. It has been proven that the set of all
projections in a $AW^{\ast }$-algebra forms a complete lattice. A lot of
important properties which are always fulfilled in $W^{\ast }$-algebras hold
as well in the $AW^{\ast }-$algebras (polar decomposition, properties of\
partial isometries, etc.). A regular construction of measurable operators
affiliated to a finite $AW^{\ast }-$algebra has been done by Berberian. This
construction is executed in an elegant algebraic manner and the output can
be interpreted as an algebra of unbounded linear operators satisfying the
annihilator condition. All results mentioned above can be found in \cite%
{Berberian72}.

Rickart $^{\ast }$-rings are $\sigma $-analogues of the Baer $^{\ast }$%
-rings. Similarly to the Baer $^{\ast }$-rings, Rickart $^{\ast }$-rings are
defined by an annihilator condition, but in another, more \textquotedblright
modest\textquotedblright\ way:

\begin{definition}
A Rickart $^{\ast }$-ring is a $^{\ast }$-ring\ $A$ such that for any
element $x$ of $A$ the right annihilator of $x$ is the principle right ideal
generated by a projection.
\end{definition}

The structure of Rickart $C^{\ast }$-algebras was studied in details in the
papers of Kaplansky, Berberian, Maeda, Ara and the first author of the
present paper (see \cite{Ara89}-\cite{Goldstein95}). It turns out so that
the basic properties such as the equivalence of the right and left
projections, the polar decomposition and the $\sigma $-normality also hold
in Rickart $C^{\ast }$-algebras (see \cite{Ara89}, \cite{AraGoldstein93}, 
\cite{Berberian72}).

Similarly to the $AW^{\ast }$-case, a finite Rickart $C^{\ast }$-algebra $A$
enjoys a regular embedding into a ring of measurable operators (see \cite%
{AraGoldstein93} and for more details \cite{Goldstein95}). This ring of
measurable operators plays a crucial role in the proof of the polar
decomposition property in Rickart $C^{\ast }$-algebras.

In this paper we define more general objects which can be also interpreted
as the algebra of unbounded operators endowed with Rickart's annihilator
condition. In the next section of the present paper we define $RO^{\ast }$%
-algebras and establish their basic properties. Notice that in our axiomatic
and technic we were strongly influenced by Chilin's work devoted to $%
BO^{\ast }$algebras (see \cite{Chilin80}). In section 3 we prove that the
algebra of all bounded elements $M$ of a $RO^{\ast }$-algebra can be endowed
with $C^{\ast }$-norm and that $M$ is complete in that norm. Section 4 is
devoted to commutative $RO^{\ast }$-algebras and their Spectral theory.

\section{Rickart ordered $^{\ast }$-algebras}

Let $T\,$be a Rickart $^{\ast }$-algebra, 
\begin{equation*}
T_{h}=\{x\in T:x^{\ast }=x\},
\end{equation*}%
\begin{equation*}
K=\{x=\overset{n}{\underset{k=1}{\sum }}x_{k}^{\ast }x_{k}\}:\;x_{k}\in T\},
\end{equation*}%
and 
\begin{equation*}
K=\left\{ x\in T_{h}:x=\overset{n}{\underset{i=1}{\sum }}x_{i}^{\ast }x_{i},%
\text{ for }n\in \mathbb{N},\text{ }x_{i}\in T\right\} .
\end{equation*}

We say that $T\;$satisfies the \textit{positive square root axiom}\ (PSR) if
for any $x\in K$ there exists%
\begin{equation*}
y\in K\cap \left\{ x\right\} ^{\prime \prime },
\end{equation*}%
such that 
\begin{equation*}
y^{2}=x,
\end{equation*}%
where $x^{\prime \prime }$ is a bicommutant of the element $x.$

\begin{proposition}
Let $T\;$be a Rickart\ $^{\ast }$-algebra satisfying the (PSR)-axiom. Then $%
K $ is a proper cone.
\end{proposition}

\begin{proof}
The properties 
\begin{equation*}
K+K\subset K,
\end{equation*}%
and 
\begin{equation*}
\lambda K\subset K,
\end{equation*}%
where $\lambda \geq 0,$ follow immediately from the definition of $K.$

Now, let\ 
\begin{equation*}
x\in K\cap (-K).
\end{equation*}%
Then$\ $%
\begin{equation*}
x=u^{\ast }u=-v^{\ast }v.
\end{equation*}%
Thus it follows that 
\begin{equation*}
u^{\ast }u+v^{\ast }v=\mathbf{0}.
\end{equation*}%
There exist $y,z\in T_{h}\;$such that$\ $%
\begin{equation*}
y^{2}=u^{\ast }u,
\end{equation*}%
\begin{equation*}
z^{2}=v^{\ast }v,
\end{equation*}%
and \ 
\begin{equation*}
y^{2}+z^{2}=\mathbf{0}.
\end{equation*}%
Therefore, 
\begin{equation*}
y^{3}+yz^{2}=\mathbf{0},
\end{equation*}%
\begin{equation*}
yz^{2}=-y^{3},
\end{equation*}%
and \ 
\begin{equation*}
yz^{2}=z^{2}y.
\end{equation*}%
In just the same way we conclude that\ 
\begin{equation*}
y^{2}z=zy^{2}.
\end{equation*}%
Hence,%
\begin{equation*}
y^{2}z^{2}=yz^{2}y=(zy^{\ast })(zy)\in K.
\end{equation*}%
There exists%
\begin{equation*}
a\in \left\{ y^{2}z^{2}\right\} ^{\prime \prime },
\end{equation*}%
\ such that \ 
\begin{equation*}
a=a^{\ast }\text{ and }a^{2}=y^{2}z^{2}.
\end{equation*}%
Note that\ 
\begin{equation*}
y\in \left\{ y^{2}z^{2}\right\} ^{\prime \prime },
\end{equation*}%
and 
\begin{equation*}
ya=ay.
\end{equation*}%
Further,%
\begin{equation*}
0=y^{2}(y^{2}+z^{2})=y^{2}+a^{2}=(y^{2}+ia)(y^{2}-ia).
\end{equation*}%
Consequently,%
\begin{equation*}
y^{2}+ia=\mathbf{0},\text{ and }x=y^{2}=\mathbf{0}.
\end{equation*}
\end{proof}

\begin{remark}
The cone $K$ defines \ a partial order on the algebra $T$ :%
\begin{equation*}
x\leq y\overset{def}{\Longleftrightarrow }y-x\in K.
\end{equation*}
\end{remark}

\begin{definition}
We say that $T$ satisfies the Fisher-Riesz axiom (FR) if the following holds:

If a sequence $\{x_{n}\}\subset T_{h}$ satisfies conditions\ 
\begin{equation*}
\mathbf{0}\leq x_{n}\leq \varepsilon _{n}\mathbf{1,}
\end{equation*}%
such that 
\begin{equation*}
\underset{k=1}{\overset{\infty }{\sum }}\varepsilon _{k}<\infty ,
\end{equation*}%
then there exists 
\begin{equation*}
\sup \underset{k=1}{\overset{n}{\sum }}x_{k}\in T_{h}.
\end{equation*}
\end{definition}

\begin{definition}
Rickart $^{\ast }$-algebra $T$\ is called $RO^{\ast }$-algebra if\ $T$\
satisfies the axioms (PSR) and (FR).
\end{definition}

\begin{definition}
Let $T$ be a $RO$ $^{\ast }$-algebra. A $^{\ast }$-subalgebra\ $B\;$of $T$
is called a $RO^{\ast }$-subalgebra, if $B$ is $RO^{\ast }$-algebra and 
\begin{equation*}
RP_{B}(x)=RP_{T}(x),
\end{equation*}%
for all $x\in B.$
\end{definition}

\begin{remark}
Recall that $BO^{\ast }$-algebra is Baer $\ast $-algebra satisfying the
axioms (PSR) and (FR) (see \cite{Chilin80} for details).
\end{remark}

\begin{example}
Let 
\begin{equation*}
X=[0,1],
\end{equation*}%
and let $F$ be a $\sigma $-algebra of $X$ that contains each subset 
\begin{equation*}
A\subseteq X,
\end{equation*}%
such that either $A$ or $X\backslash A$ is countable.\ Define $T$ to be a $%
^{\ast }$-algebra of all measurable complex functions on $(X,F)$. It is
easily seen that $T$ is a $RO^{\ast }$-algebra which is not $BO^{\ast }$%
-algebra.
\end{example}

\section{The algebra of bounded elements in a $RO^{\ast }$-algebra}

The order properties which are axiomatically defined on $RO^{\ast }$-algebra 
$T$ allow us to specify a set of bounded elements of $T$.

In arbitrary $RO^{\ast }$-algebra $T\,$\ we will extract some $C^{\ast }$%
-subalgebra which contains all of bounded elements of $T$.

\begin{definition}
An element 
\begin{equation*}
a+ib,
\end{equation*}%
$(a,b\in T_{h})\;$is \textit{bounded, }if there exists $\lambda \geq 0$ such
that 
\begin{equation*}
-\lambda \mathbf{1\leq }a,b\leq \lambda \mathbf{1.}
\end{equation*}
\end{definition}

\begin{proposition}
The set of all bounded elements of a$\ RO^{\ast }$-algebra is a $RO^{\ast }$%
-subalgebra. If 
\begin{equation*}
x^{\ast }x\leq \mathbf{1,}
\end{equation*}%
then 
\begin{equation*}
xx^{\ast }\leq \mathbf{1}.
\end{equation*}
\end{proposition}

\begin{proof}
Let us denote by\ $M$ the set of all bounded elements of $T$. Let $x,y\in M,$%
\begin{equation*}
x=a+ib,\;y=c+id,
\end{equation*}%
\begin{equation*}
-\lambda \mathbf{1\leq }a,b\leq \lambda \mathbf{1,}
\end{equation*}%
\begin{equation*}
\mathbf{-}\mu \mathbf{1\leq }c,d\leq \mu \mathbf{1,}
\end{equation*}%
for some $\lambda ,\mu \geq 0$. Obviously, 
\begin{equation*}
x+y,x^{\ast },\alpha x\in M,
\end{equation*}%
where $\alpha \in \mathbb{C}.$ We will prove that $xy$ is also bounded. It
will be sufficient to prove it for the case 
\begin{equation*}
x,y\in T_{h},
\end{equation*}%
i.e. 
\begin{equation*}
x=a,\;y=c.
\end{equation*}%
If \ 
\begin{equation*}
x,y\geq \mathbf{0},
\end{equation*}%
then, since 
\begin{equation*}
\sqrt{x}\in \{x\}^{\prime \prime },
\end{equation*}%
\begin{equation*}
\sqrt{y}\in \{y\}^{\prime \prime },
\end{equation*}%
we have 
\begin{equation*}
xy\geq \mathbf{0}.
\end{equation*}%
Therefore, the inequalities 
\begin{equation*}
\lambda \mathbf{1+}x\geq \mathbf{0}\text{ and }\lambda \mathbf{1-}x\geq 
\mathbf{0},
\end{equation*}%
imply 
\begin{equation*}
x^{2}\leq \lambda ^{2}\mathbf{1.}
\end{equation*}%
In just the same way 
\begin{equation*}
y^{2}\leq \mu ^{2}\mathbf{1.}
\end{equation*}%
Since 
\begin{equation*}
(x+y)^{2}\geq \mathbf{0},
\end{equation*}%
we have%
\begin{equation*}
-(x^{2}+y^{2})\leq (xy+yx)\leq (x^{2}+y^{2}).
\end{equation*}%
Consequently,%
\begin{equation*}
-(\lambda ^{2}+y^{2})\mathbf{1}\leq (xy+yx)\leq (\lambda ^{2}+\mu ^{2})%
\mathbf{1}.
\end{equation*}%
Similarly,%
\begin{equation*}
-(\lambda ^{2}+y^{2})\mathbf{1}\leq i(xy-yx)\leq (\lambda ^{2}+\mu ^{2})%
\mathbf{1}.
\end{equation*}%
On the other hand,%
\begin{equation*}
xy=\frac{(xy+yx)}{2}-i\left( \frac{\mathbf{1}}{2i}\right) (xy-yx),
\end{equation*}%
hence, $xy\in M.$

Thus,\ $M$ is a $^{\ast }$-subalgebra of $T.$

Let now $x\in M,\;$%
\begin{equation*}
R_{M}(x)=\left\{ y\in M:xy=\mathbf{0}\right\} ,
\end{equation*}%
be a right annihilator of $x.$ There exists a projection 
\begin{equation*}
e\in M,
\end{equation*}%
such that 
\begin{equation*}
R(x)=\left\{ x\in T:xy=\mathbf{0}\right\} =eT.
\end{equation*}%
It's clear that 
\begin{equation*}
R_{M}(x)=R(x)\cap M.
\end{equation*}%
On the other hand, 
\begin{equation*}
\mathbf{0}\leq e\leq \mathbf{1,}
\end{equation*}%
therefore $e\in M,\;$and 
\begin{equation*}
RP_{M}(x)=RP_{T}(x).
\end{equation*}%
If 
\begin{equation*}
x\in M\cap K,
\end{equation*}%
then there exists 
\begin{equation*}
y\in K\cap \left\{ x\right\} ^{\prime \prime },
\end{equation*}%
such that 
\begin{equation*}
y^{2}=x.
\end{equation*}%
Since 
\begin{equation*}
-\mathbf{1}-y^{2}\leq 2y\leq \mathbf{1+}y^{2},
\end{equation*}%
we have 
\begin{equation*}
-\frac{(\mathbf{1}+x)}{2}\leq y\leq \frac{(\mathbf{1}+x)}{2},
\end{equation*}%
i.e. $y$ is a bounded element. Thus, $M$ satisfies (PSR).

Let now a sequence $\left\{ x_{n}\right\} _{n\in N}\;$satisfies conditions 
\begin{equation*}
\mathbf{0}\leq x_{n}\leq \varepsilon _{n}\mathbf{1,}
\end{equation*}%
and \ 
\begin{equation*}
\sum_{n=1}^{\infty }\varepsilon _{n}<+\infty .
\end{equation*}%
Obviously, 
\begin{equation*}
\sum_{n=1}^{k}x_{n}\leq \left( \sum_{n=1}^{k}\varepsilon _{n}\right) \mathbf{%
1\leq }\left( \sum_{n=1}^{\infty }\varepsilon _{n}\right) \mathbf{1,}
\end{equation*}%
for all $k\in \mathbb{N}$. Then$\ $%
\begin{equation*}
\mathbf{0}\leq \sup \sum_{n=1}^{k}x_{n}\leq \left( \sum_{n=1}^{\infty
}\varepsilon _{n}\right) \mathbf{1,}
\end{equation*}%
i.e., the algebra $M$ satisfies the axiom (RF).
\end{proof}

\begin{theorem}
Let $T$ be a $RO^{\ast }$-algebra, $M$ denotes a subalgebra of all bounded
elements of $T.$ Then there exists a norm on $M$ \ such that $M$ is
(relatively to this norm)$\,$a $C^{\ast }$-algebra.
\end{theorem}

\begin{proof}
Define 
\begin{equation*}
\left\Vert x\right\Vert =\inf \left\{ \lambda \geq 0:-\lambda \mathbf{1\leq }%
x\leq \lambda \mathbf{1}\right\} .
\end{equation*}%
One can easily check that $\left\Vert \cdot \right\Vert $ is a norm on $%
M_{h}.$ We will prove that \ $\left( M_{h},\left\Vert \cdot \right\Vert
\right) $ is a Banach space. Let 
\begin{equation*}
\left\{ x_{n}\right\} _{n\in N}\subset M_{h},
\end{equation*}%
\begin{equation*}
\sum_{n=1}^{\infty }x_{n}<+\infty .
\end{equation*}%
Let us denote%
\begin{equation*}
a_{n}=\left\Vert x_{n}\right\Vert \mathbf{1}+x_{n}\mathbf{,\;}%
b_{n}=\left\Vert x_{n}\right\Vert \mathbf{1}-x_{n}.
\end{equation*}%
Then 
\begin{equation*}
x_{n}=\frac{1}{2}\left( a_{n}-b_{n}\right) ,
\end{equation*}%
\begin{equation*}
\left\Vert a_{n}\right\Vert \leq 2\left\Vert x_{n}\right\Vert ,
\end{equation*}%
\begin{equation*}
\left\Vert b_{n}\right\Vert \leq 2\left\Vert x_{n}\right\Vert .
\end{equation*}%
Therefore%
\begin{equation*}
\mathbf{0}\leq a_{n}\leq 2\left\Vert x_{n}\right\Vert \mathbf{1},\;\mathbf{0}%
\leq b_{n}\leq 2\left\Vert x_{n}\right\Vert \mathbf{1.}
\end{equation*}

Because of axiom the\ (RF) axiom,\ there exist$\ $%
\begin{equation*}
\sup_{n}\sum_{n=1}^{k}a_{n}=a,\text{ and}\ \sup_{n}\sum_{n=1}^{k}b_{n}=b.
\end{equation*}

Since 
\begin{equation*}
\sum_{k=1}^{n}a_{k}\leq 2\sum_{k=1}^{\infty }\left\Vert x_{k}\right\Vert 
\mathbf{1},
\end{equation*}%
we have inequalities\ 
\begin{equation*}
\mathbf{0}\leq a\leq 2\sum_{k=1}^{\infty }\left\Vert x_{k}\right\Vert 
\mathbf{1},
\end{equation*}%
i.e. $a\in M$.\ In just the same way we obtain the fact that $b\in M$.

We will prove now that $\ $%
\begin{equation*}
||a-\sum_{n=1}^{k}a_{n}||\rightarrow 0,\text{ when}\ k\rightarrow \infty .
\end{equation*}%
In fact,%
\begin{equation*}
\left\Vert a-\sum_{n=1}^{k}a_{n}\right\Vert =\underset{s\geq k}{\sup }\left(
\sum_{i=1}^{s}a_{n}-\sum_{n=1}^{k}a_{n}\right) =\underset{s\geq k}{\sup }%
\sum_{n=k+1}^{s}a_{n}
\end{equation*}%
Note that 
\begin{equation*}
\mathbf{0}\leq \sum_{n=k+1}^{s}a_{n}\leq \left( 2\sum_{n=k+1}^{\infty
}||x_{n}||\right) \mathbf{1.}
\end{equation*}%
Hence 
\begin{equation*}
a-\sum_{n=1}^{k}a_{n}\leq \left( 2\sum_{n=k+1}^{\infty }||x_{n}||\right) 
\mathbf{1,}
\end{equation*}%
and 
\begin{equation*}
\left\Vert a-\sum_{n=1}^{k}a_{n}\right\Vert \leq 2\sum_{n=k+1}^{\infty
}||x_{n}||\rightarrow 0,
\end{equation*}%
as\ $k\rightarrow \infty $. In the analogous way 
\begin{equation*}
\left\Vert b-\sum_{n=1}^{k}b_{n}\right\Vert \rightarrow 0,
\end{equation*}%
as\ $k\rightarrow \infty $.

Now,\ 
\begin{equation*}
\sum_{n=1}^{k}x_{n}=\frac{1}{2}\sum_{n=1}^{k}\left( a_{n}-b_{n}\right)
\rightarrow \frac{1}{2}\left( a-b\right) ,
\end{equation*}%
as\ $k\rightarrow \infty $. Thus, $\left( M_{h},\left\Vert \cdot \right\Vert
\right) $ is a Banach space.

The next claim is to prove that the cone 
\begin{equation*}
K\cap M,
\end{equation*}%
is closed in the norm\ $\left\Vert \cdot \right\Vert $ Let $x\in K,$ and%
\begin{equation*}
\left\Vert x\right\Vert \leq 1.
\end{equation*}%
Then 
\begin{equation*}
\mathbf{0}\leq x\leq \mathbf{1},
\end{equation*}%
\begin{equation*}
0\leq \mathbf{1}-x\leq \mathbf{1},
\end{equation*}%
\begin{equation*}
\left\Vert \mathbf{1}-x\right\Vert \leq 1.
\end{equation*}%
Conversely, if $x\in M_{h},$%
\begin{equation*}
\left\Vert x\right\Vert \leq 1,
\end{equation*}%
\begin{equation*}
\left\Vert \mathbf{1}-x\right\Vert \leq 1,
\end{equation*}%
then 
\begin{equation*}
\mathbf{1}-x\leq \mathbf{1},
\end{equation*}%
\begin{equation*}
x\geq \mathbf{0}.
\end{equation*}%
Thus,%
\begin{equation*}
\left\{ x\in K:\left\Vert x\right\Vert \leq 1\right\} =\left\{ x\in
M_{h}:\left\Vert x\right\Vert \leq 1,\;\left\Vert \mathbf{1}-x\right\Vert
\leq 1\right\} .
\end{equation*}%
Let%
\begin{equation*}
x_{n}\geq \mathbf{0},
\end{equation*}%
\begin{equation*}
\underset{n\rightarrow \infty }{\lim }x_{n}=x.
\end{equation*}%
Then 
\begin{equation*}
\underset{n\rightarrow \infty }{\lim }\left\Vert x_{n}\right\Vert
=\left\Vert x\right\Vert .
\end{equation*}%
It is easy to see that there exists 
\begin{equation*}
\lambda >0,
\end{equation*}%
such that 
\begin{equation*}
\left\Vert \lambda x_{n}\right\Vert \leq 1.
\end{equation*}%
Since 
\begin{equation*}
\lambda x_{n}\geq \mathbf{0},
\end{equation*}%
we have that 
\begin{equation*}
\left\Vert \mathbf{1}-\lambda x_{n}\right\Vert \leq 1.
\end{equation*}%
Obviously, 
\begin{equation*}
\left\Vert \lambda x\right\Vert \leq 1,
\end{equation*}%
and%
\begin{equation*}
\left\Vert \mathbf{1}-\lambda x\right\Vert \leq 1.
\end{equation*}%
Therefore 
\begin{equation*}
\lambda x\in K,
\end{equation*}%
and 
\begin{equation*}
x\in K.
\end{equation*}%
Thus, 
\begin{equation*}
K\cap M_{h},
\end{equation*}%
is closed in $\left( M_{h},\left\Vert \cdot \right\Vert \right) .$ Assume
now that 
\begin{equation*}
x\in M,
\end{equation*}%
\begin{equation*}
x\neq 0,
\end{equation*}%
and%
\begin{equation*}
y=-x^{\ast }x.
\end{equation*}%
Then 
\begin{equation*}
y\neq \mathbf{0},
\end{equation*}%
and 
\begin{equation*}
y\notin K.
\end{equation*}%
By the Hahn-Banach Theorem there exists a continuous functional $\varphi $\
on $M_{h}$ which is non-negative on\ $K$ and 
\begin{equation*}
\varphi (y)<0.
\end{equation*}%
We now extend the functional $\varphi $ up to a functional on $M$ \ in the
usual linear manner: 
\begin{equation*}
\varphi (a+ib)=\varphi (a)+i\varphi (b),
\end{equation*}%
for any $a,b\in M_{h}$. Therefore $\varphi \;$is the state on $M,$ i.e.%
\begin{equation*}
\varphi (x^{\ast }x)>0,
\end{equation*}%
and we can put 
\begin{equation*}
\varphi (\mathbf{1})=1.
\end{equation*}

Let $\pi _{\varphi }$ be a representation of a $^{\ast }$-algebra $M$ in the
algebra of all bounded linear operators on some Hilbert space $H_{\varphi }$
(the Gelfand-Naimark-Segal construction). It is well known that for all\ $%
x\in M$ the following inequality is valid:%
\begin{equation*}
\varphi (y^{\ast }x^{\ast }xy)\leq \left\Vert x^{\ast }x\right\Vert \varphi
(y^{\ast }y).
\end{equation*}%
Hence$\ \pi _{\varphi }(x)$ is a bounded linear operator on $H_{\varphi }$,
that is, $\pi _{\varphi }$ is a $^{\ast }$-representation of $\ M$ into the $%
^{\ast }$-algebra $B(H_{\varphi })$.

Now, let 
\begin{equation*}
\pi =\underset{\varphi \in \Phi }{\bigoplus }\pi _{\varphi },
\end{equation*}%
be a direct sum of the all states on$\ M$, where $\Phi $ denotes the set of
all states on $M$. Then $\Phi $\ is a faithful representation of $M$ into
the algebra $B\left( \underset{\varphi \in \Phi }{\bigoplus }H_{\varphi
}\right) $. It is clear now that the norm 
\begin{equation*}
q(x)=\left\Vert \pi (x)\right\Vert ,
\end{equation*}%
induces the norm $\left\Vert \cdot \right\Vert $\ on\ $M_{h}$.
\end{proof}

\section{Commutative $RO^{\ast }$-algebras and the Spectral Theorem}

In this section we describe the structure of commutative $RO^{\ast }$%
-algebras and prove the Spectral Theorem for self-adjoint element of a $%
RO^{\ast }$-algebra.

\begin{theorem}
Let $T$ be a commutative $RO^{\ast }$-algebra. Then $T_{h}$ is a
conditionally $\sigma $-complete lattice.
\end{theorem}

\begin{proof}
Assume that $x\in T_{h}$ and $x$ is non-comparable with $\mathbf{0}$. There
exists 
\begin{equation*}
y\geq \mathbf{0,}
\end{equation*}%
such that 
\begin{equation*}
y^{2}=x^{2}.
\end{equation*}

Put 
\begin{equation*}
a=\frac{1}{2}\left( y+x\right) ,
\end{equation*}%
and%
\begin{equation*}
b=\frac{1}{2}\left( y-x\right) .
\end{equation*}%
Since $x$ is non-comparable with $\mathbf{0}$, we can conclude that $a,b\neq 
\mathbf{0}$.

Note that 
\begin{equation*}
b\in r(b+x),
\end{equation*}%
and%
\begin{equation*}
a\in r(x-a).
\end{equation*}%
Because $T$ is Rickart $^{\ast }$-algebra, there exist projections $e$ and $%
f,$ such that 
\begin{equation*}
r(x+b)=eT,
\end{equation*}%
and%
\begin{equation*}
r(a-x)=fT.
\end{equation*}%
Obviously, 
\begin{equation*}
xe=ye,
\end{equation*}%
\begin{equation*}
xf=-yf,
\end{equation*}%
and 
\begin{equation*}
\left( \mathbf{1}-e\right) \left( \mathbf{1}-f\right) =\mathbf{0}.
\end{equation*}%
Therefore, 
\begin{equation*}
f-ef=\mathbf{1}-e,
\end{equation*}%
and 
\begin{equation*}
ye-x=-x(1-e)=-x(f-ef)=y(f-ef)\geq \mathbf{0}.
\end{equation*}%
Thus, 
\begin{equation*}
ye\geq x.
\end{equation*}%
Now, let 
\begin{equation*}
d\in T_{h},
\end{equation*}%
\begin{equation*}
d\geq \mathbf{0,}
\end{equation*}%
and 
\begin{equation*}
d\geq x.
\end{equation*}%
Then 
\begin{equation*}
d=de+d(1-e)\geq ede\geq exe=ye.
\end{equation*}%
This yields 
\begin{equation*}
ye=\sup \left\{ x,\mathbf{0}\right\} =x\vee \mathbf{0}.
\end{equation*}%
We have showed that $T_{h}$ is a vector lattice.

To prove the conditional $\sigma $-completeness let us consider an
increasing sequence $\left\{ x_{n}\right\} $, where 
\begin{equation*}
x_{n}\geq \mathbf{0}
\end{equation*}%
and 
\begin{equation*}
x_{n}\leq v.
\end{equation*}%
The element 
\begin{equation*}
w=\sqrt{v^{2}+\mathbf{1}}+v,
\end{equation*}%
has an inverse 
\begin{equation*}
w^{-1}=\sqrt{v^{2}+\mathbf{1}}-v.
\end{equation*}%
The sequence $\left\{ w^{-1}x_{n}w^{-1}\right\} $ is increasing and bounded
by $w^{-1}$. It is easy to see that 
\begin{equation*}
w^{-1}\leq \mathbf{1}.
\end{equation*}%
Thus 
\begin{equation*}
w^{-1}x_{n}w^{-1}\in M_{h}.
\end{equation*}%
By Theorem 1 $M$ is a Rickart $C^{\ast }$-algebra, consequently, $M_{h}$ is
a conditionally $\sigma $-complete lattice (see, for example, \cite{Ara89}).
Therefore, there exists 
\begin{equation*}
x=\sup \left\{ w^{-1}x_{n}w^{-1}\right\} .
\end{equation*}%
It is clear that 
\begin{equation*}
wxw=\sup \left\{ x_{n}\right\} ,
\end{equation*}%
belongs to $T_{h}$.
\end{proof}

\begin{corollary}
Let $T$ be a commutative $RO^{\ast }$-algebra. Then $T_{h}$ is $K_{\sigma }$%
-space.
\end{corollary}

\begin{theorem}
Let $T$ be a $RO^{\ast }$-algebra and let $B$\ be some maximal commutative
subalgebra of $T$. Then $B$\ is a $RO^{\ast }$-subalgebra of $T$.
\end{theorem}

\begin{proof}
The proof that $B$ is a Rickart $\ast $-algebra is analogous to \cite%
{Chilin80}. Further, let 
\begin{equation*}
x\in K_{B}=\left\{ \overset{n}{\underset{i=1}{\sum }}x_{i}^{\ast
}x_{i}:x_{i}\in B,\text{ where}\;i=1,2,...,n\right\} .
\end{equation*}%
There exists 
\begin{equation*}
y\in K_{B}\cap \left\{ x\right\} ^{\prime \prime },
\end{equation*}%
such that 
\begin{equation*}
y^{2}=x.
\end{equation*}%
Thus the algebra $B$ satisfies the(PSR) axiom. To prove the (RF) axiom, we
take 
\begin{equation*}
x_{n}\in B,
\end{equation*}%
\begin{equation*}
\mathbf{0}\leq x_{n}\leq \varepsilon _{n}\mathbf{1,}
\end{equation*}%
and%
\begin{equation*}
\underset{k=1}{\overset{\infty }{\sum }}\varepsilon _{k}<\infty .
\end{equation*}%
One can see there exists $\ $%
\begin{equation*}
x=\sup \underset{k=1}{\overset{n}{\sum }}x_{k},
\end{equation*}%
in\ $T_{h}$. Since 
\begin{equation*}
x\leq \left( \underset{k=1}{\overset{\infty }{\sum }}\varepsilon _{k}\right) 
\mathbf{1,}
\end{equation*}%
we can conclude that $x$ belongs to the Rickart $C^{\ast }$-algebra of
bounded elements of $T$. We have seen above (Theorem 1) that 
\begin{equation*}
\left\Vert x-s_{n}\right\Vert \rightarrow 0,
\end{equation*}%
as $n\rightarrow \infty ,$ where$\ $%
\begin{equation*}
s_{n}=\underset{k=1}{\overset{n}{\sum }}x_{k},
\end{equation*}%
and therefore 
\begin{equation*}
x\in \left( B\cap M\right) ^{\prime }.
\end{equation*}%
Now let 
\begin{equation*}
a\in B_{h},
\end{equation*}%
and%
\begin{equation*}
a\geq \mathbf{0}.
\end{equation*}%
Then 
\begin{equation*}
y=\sqrt{x^{2}+\mathbf{1}}-x\geq a+\mathbf{1,}
\end{equation*}%
and $y$ is invertible. Notice that 
\begin{equation*}
\mathbf{0}\leq y^{-1}\leq \mathbf{1,}
\end{equation*}%
and 
\begin{equation*}
\mathbf{0}\leq y^{-1}ay^{-1}\leq y^{-1}.
\end{equation*}%
Thus 
\begin{equation*}
y^{-1}ay^{-1}\in B\cap M.
\end{equation*}%
Since 
\begin{equation*}
y^{-1}\in B\cap M,
\end{equation*}%
we have 
\begin{equation*}
xy^{-1}=y^{-1}x.
\end{equation*}%
Therefore 
\begin{equation*}
y^{-1}xay^{-1}=y^{-1}axy^{-1},
\end{equation*}%
and%
\begin{equation*}
xa=ax.
\end{equation*}%
As any element $c\in B$ is linear combination of positive ones, we obtain
that 
\begin{equation*}
xc=cx.
\end{equation*}%
Since $B$ is a maximal one we can conclude that $x\in B$.

To finish the proof, notice that 
\begin{equation*}
x=\sup s_{n},
\end{equation*}%
in $T_{h}.$ Hence 
\begin{equation*}
x=\sup s_{n},
\end{equation*}%
in $B_{h}$. Thus $B$ is a $RO^{\ast }$-algebra.
\end{proof}

Now we can proceed to formulate the main result of this section.

\begin{theorem}[Spectral Theorem]
Let $T$ be a $RO^{\ast }$-algebra, $x\in T_{h}$. There exists a unique
family of projections $\left\{ e_{\lambda }\right\} _{\lambda \in \mathbf{R}%
} $ satisfying the following properties:

(a) 
\begin{equation*}
e_{\lambda }\leq e_{\mu }\text{ if }\lambda \leq \mu ;
\end{equation*}

(b) 
\begin{equation*}
\inf e_{\lambda }=\mathbf{0};
\end{equation*}

(c) 
\begin{equation*}
\sup e_{\lambda }=\mathbf{1};
\end{equation*}

(d) 
\begin{equation*}
\underset{\mu <\lambda }{\sup }e_{\mu }=e_{\lambda };
\end{equation*}

(e) 
\begin{equation*}
e_{\lambda }\cdot x\leq \lambda e_{\lambda },\;e_{_{\lambda }^{\bot }}\cdot
x\geq \lambda e_{_{\lambda }^{\bot }}.
\end{equation*}

Moreover, \ for all $\varepsilon >0$ there exists $\delta >0$ such that 
\begin{equation*}
\left\Vert x-\overset{n}{\underset{i=1}{\sum }}\xi _{i}\left( \lambda
_{i}-\lambda _{i-1}\right) \right\Vert <\varepsilon ,
\end{equation*}%
for any partition 
\begin{equation*}
\left\{ \lambda _{i}\right\} _{i=0}^{\infty },
\end{equation*}%
of the real line with 
\begin{equation*}
\sup \left( \lambda _{i}-\lambda _{i-1}\right) <\delta ,
\end{equation*}%
and 
\begin{equation*}
\xi _{i}\in \left[ \lambda _{i-1},\lambda _{i}\right] .
\end{equation*}
\end{theorem}

\begin{proof}
One obtains the proof by using the classical proof of the Spectral Theorem
for $K_{\sigma }$-spaces in \cite{Vulikh61}, and previous results of this
section.
\end{proof}

\end{document}